\documentclass[11pt]{amsart}
\usepackage{amsmath,amsfonts,amssymb,amsthm,txfonts,pxfonts,amscd}
\usepackage{geometry}
\usepackage{color}
\usepackage[T1]{fontenc}
\usepackage{a4wide}
\usepackage{bm}
\usepackage[mathcal]{euscript}

\newcommand{\comment}[1]{}

\newcommand{\mathnotation}[2]{\newcommand{#1}{\ensuremath{#2}}}

\mathnotation{\deltthree}{\Delta_3}
\mathnotation{\deltfour}{\Delta_4}
\mathnotation{\deltfive}{\Delta_5}
\mathnotation{\deltsix}{\Delta_6}

\mathnotation{\ldef}{\mathrel{\raisebox{.069ex}{:}\!\!=}}
\mathnotation{\rdef}{\mathrel{=\!\!\raisebox{.069ex}{:}}}
\mathnotation{\ho}{\varphi}
\mathnotation{\hoii}{\psi}
\mathnotation{\surf}{S}		
\mathnotation{\fol}{\mathcal{F}}
\mathnotation{\folu}{\fol^{\mathrm{u}}}
\mathnotation{\fols}{\fol^{\mathrm{s}}}
\mathnotation{\id}{Id}
\mathnotation{\tmeas}{\mu}
\mathnotation{\tmeasu}{\tmeas^{\mathrm{u}}}
\mathnotation{\tmeass}{\tmeas^{\mathrm{s}}}
\mathnotation{\Nsing}{N}
\mathnotation{\prongs}{\#\text{prongs}}

\mathnotation{\aaa}{a}				
\newcommand{\infloop}[1]{}			
\mathnotation{\fp}{p}
\mathnotation{\fq}{p'}
\mathnotation{\fpp}{S}
\mathnotation{\fpq}{Q}
\mathnotation{\xx}{X}				%
\mathnotation{\po}{m} 

\mathnotation{\Fix}{\mathrm{Fix}}

\newcommand{\Mathematica}{\textsl{Mathematica}}

\newcommand{\R}{\mathbb{R}}

\newcommand{\Z}{\mathbb{Z}}

\newcommand{\Tr}{Tr}

\newcommand\pA{\phi}

\newtheorem{Theorem}{Theorem}[section]
\newtheorem{Corollary}[Theorem]{Corollary}
\newtheorem{Proposition}[Theorem]{Proposition}
\newtheorem{Lemma}[Theorem]{Lemma}

\newtheorem{Remark}[Theorem]{Remark}
\newtheorem{Example}[Theorem]{Example}
\newtheorem{Definition}[Theorem]{Definition}
\newtheorem{Convention}[Theorem]{Convention}

\newtheorem*{NoNumberTheorem}{Theorem}
\newtheorem*{NoNumberLemma}{Lemma}

\begin{document}

\title[Braids with small dilatations]
{On the minimum dilatation of braids on punctured discs}
\author{Erwan Lanneau, Jean-Luc Thiffeault}

\address{
Centre de Physique Th\'eorique (CPT), UMR CNRS 6207 \newline
Universit\'e du Sud Toulon-Var and \newline
F\'ed\'eration de Recherches des Unit\'es de
Math\'ematiques de Marseille \newline
Luminy, Case 907, F-13288 Marseille Cedex 9, France
}
\email{erwan.lanneau@cpt.univ-mrs.fr}

\address{
Department of Mathematics \newline Van Vleck Hall, 480 Lincoln Drive \newline
 University of Wisconsin - Madison, WI 53706, USA}

\email{jeanluc@math.wisc.edu}
\date{\today}

\subjclass[2000]{Primary: 37D40. Secondary: 37E30}
\keywords{pseudo-Anosov homeomorphisms, braids, foliations}

\begin{abstract}
We find the minimum dilatation of pseudo-Anosov braids on
$n$-punctured discs for~$3 \leq n \leq 8$.  This covers the results
of Song-Ko-Los ($n=4$) and Ham-Song ($n=5$). The proof is elementary, and 
uses the Lefschetz formula.
\end{abstract}

\maketitle

\section{Introduction}

There are many results on the minimum dilatation of pseudo-Anosov
homeomorphisms.  They includes bounds and specific
examples~\cite{Penner1991, Brinkmann2004, Leininger2004, Hironaka2006,
  Minakawa2006, Thiffeault2006, Venzke_thesis, Tsai2009, Aaber2010,
  Kin2010b, Kin2010a} as well as known values on closed and punctured
surfaces~\cite{Zhirov1995, Song2002, Song2005, ChoHam2008, Ham2007,
  Hironaka2009_preprint, LanneauThiffeault2010}.  For the punctured
discs, the case with three punctures is classical.  Discs with four
and five punctures were solved by~\cite{Song2002} and~\cite{Ham2007}
using train track automata.  In this paper we give a simple derivation
for four and five punctures, and find the least dilatation for up to
eight punctures, using methods introduced
in~\cite{LanneauThiffeault2010}.

It is well-known that the braid group $B_n$ is isomorphic to the
mapping class group $\textrm{Mod}(0,n+1)$ of the sphere with $n+1$
punctures (one of which is a marked point), that is the disc with $n$
punctures. We denote by $\sigma_i \in B_n$ the classical
generators~\cite{Birman1975}.  We shall prove

\begin{Theorem}
\label{theo:main}
For $3 \leq n \leq 8$, the minimum dilatation $\delta_n$ of
pseudo-Anosov $n$-braids is the Perron root (maximal root) of the following
polynomials:
\begin{table}[htbp]
\begin{tabular}{lllll}
\hline
n & $\delta_n \simeq$ & polynomial & braid & stratum \\
\hline
3 & 2.61803 & $\xx^2-3\xx+1$ & $\sigma_1 \sigma^{-1}_2$ & $(-1;-1^3)$ \\
4 & 2.29663 & $\xx^4-2\xx^3-2\xx+1$ & $\sigma_1 \sigma_2 \sigma^{-1}_3$ & $(-1;-1^4,1)$ \\
5 & 1.72208 & $\xx^4-\xx^3-\xx^2-\xx+1$ & $\sigma_1 \sigma_2 \sigma_3 \sigma_1 \sigma_2 \sigma_3 \sigma_4 \sigma^{-1}_3$ & $(0;-1^5,1)$ \\
6 & 1.72208 & $\xx^4-\xx^3-\xx^2-\xx+1$ &
$\sigma_2\sigma_1\sigma_2\sigma_1\left( \sigma_1 \sigma_2 \sigma_3 \sigma_4 \sigma_5 \right)^2$ & $(0;-1^5,1)$ \\
7 & 1.46557 & $\xx^7-2\xx^4-2\xx^3+1$ & $\sigma_4^{-2} (\sigma_1 \sigma_2 \sigma_3 \sigma_4 \sigma_5 \sigma_6)^2$ & $(2;-1^7,1)$ \\
8 & 1.41345 & $\xx^8-2\xx^5-2\xx^3+1$ & $\sigma_2^{-1}\sigma_1^{-1} (\sigma_1 \sigma_2 \sigma_3 \sigma_4 \sigma_5 \sigma_6 \sigma_7)^5$ & $(3,-1^8,1)$ \\
\hline
\end{tabular}
\end{table}
\end{Theorem}
The notation for strata is explained in Section~\ref{sec:strata}.
Note that for $n=6$ the pseudo-Anosov with smallest dilatation is
identical to that for $n=5$, but with a punctured degree-$1$
singularity, as conjectured by Venzke~\cite{Venzke_thesis}.  For $n=5$
and~$7$ the braids are part of a sequence described
in~\cite{Hironaka2006,Thiffeault2006,Venzke_thesis}.  For $n=8$ the
minimizing braid is the one described by Venzke~\cite{Venzke_thesis}.

In the next section, we introduce the tools we will use to prove
Theorem~\ref{theo:main}, and sketch the proof in
Section~\ref{sec:outline}. The first two cases of the theorem ($n=3$
and $n=4$) are detailed in Sections~\ref{section:n=3}
\&~\ref{section:n=4}.  For the other cases
(Sections~\ref{section:n=5}--\ref{section:n=8}) we use a computer to
find reciprocal polynomials with a Perron root less than a given
constant and for the combinatorics of the Lefschetz numbers; this is
straightforward and elementary.\footnote{These
  \Mathematica~\cite{Mathematica7} functions are included as
  Electronic Supplementary Material in the file
  \texttt{PseudoAnosovLite.m} and the example notebook
  \texttt{disc5.nb} for the disc with~$5$ punctures.}  In an appendix
we give the minimum dilatation for each stratum and provide explicit
examples of a braid realizing each minimum (for $3\le n \le 7$).

Presumably the method could be used on discs with more punctures: the
limiting steps are (i) the generation of the list of polynomials on
the generic stratum (with the most allowable degree-1 singularities);
(ii) the eliminination of polynomials `by hand' using the
combinatorics of orbits or the action on singulaties.  We also know by
examining each stratum (see appendix, Section~\ref{sec:s4dic5}) that
sometimes polynomials cannot be ruled out in this manner; in those
cases other techniques must be used, such as train track
automata~\cite{LanneauThiffeault_ttauto}.

\subsection*{Acknowledgments}

J-LT was supported by the Division of Mathematical Sciences of the US
National Science Foundation, under grant DMS-0806821. EL was supported
by the Centre National de la Recherche Scientifique under project
PICS~4170 (France--USA).

\section{Preliminaries}

For basic references on braid groups, mapping class groups, and
pseudo-Anosov homeomorphisms see for example~\cite{Birman1975,
  Fathi1979, Thurston1988}.

\subsection{Braid groups}

Let $n\geq 3$ be an integer. The braid group $B_n$ is defined by the presentation 
\[
B_n =  \bigl\langle \sigma_1,\dots,\sigma_{n-1}\ \bigl| \
\text{$\sigma_i \sigma_j=\sigma_j \sigma_i$ if $|i-j| \geq 2$ and
$\sigma_i \sigma_j \sigma_i = \sigma_j \sigma_i\sigma_j$ if $|i-j|=1$}
\bigr\rangle.
\]
The group $B_n$ is naturally identified with the group of homotopy
classes of orientation-preserving homeomorphisms of an $n$-punctured
disc, fixing the boundary pointwise \cite{Birman1975}.  One can see
this as follows. Let $\beta \in B_n$ be a geometric $n$-braid, sitting
in the cylinder $[0, 1]\times D$ with $D$ the unit disc, whose $n$
strands start at the puncture points of $\{0\}\times D$ and end at the
puncture points of $\{1\} \times D$. The braid may be considered as
the graph of the motion, as time goes from $1$ to $0$, of $n$ points
moving in the disc, starting and ending at the puncture points. It can
be proved that this motion extends to a continuous family of
homeomorphisms of the disc, starting with the identity and fixed on
the boundary at all times. The end map of this isotopy is a
homeomorphism $h: D \rightarrow D$, which is well-defined up to
isotopy fixed at the punctures and the boundary.  Conversely, given a
homeomorphism $h: D \rightarrow D$ representing some element of the
mapping class group, we want to get a geometric $n$-braid. By a
well-known trick of Alexander, any homeomorphism of a disc which fixes
the boundary is isotopic to the identity, through homeomorphisms
fixing the boundary. The corresponding braid is then just the graph of
the restriction of such an isotopy to the puncture points.  Thus there
is a canonical embedding of the braid group $B_n$ to the mapping class
group of the sphere with $n+1$ punctures
$\textrm{Mod}(0,n+1)$. Geometrically, the standard generators
$\sigma_1,\dots,\sigma_{n-1}$ are induced by right Dehn
half-twists around loops enclosing the punctures $p_{i}$ and $p_{i+1}$
(up to homotopy). \medskip

Obviously, orientation-preserving disc
homeomorphisms and orientation-preserving sphere homeomorphisms with a
fixed marked point define the same object.

\begin{Convention}
\label{convention:sphere}
All sphere homeomorphisms in this paper will fix a marked point
(regular or singular) on the sphere.
\end{Convention}

\subsection{Pseudo-Anosov homeomorphisms}

The classification theorem of Thurston~\cite{Thurston1988} asserts
that any orientation-preserving homeomorphism of a compact surface $S$
is relative-isotopic to a finite-order, reducible, or pseudo-Anosov
homeomorphism.  In this paper we are interested in the last case.  A
homeomorphism $\pA$ is pseudo-Anosov if there exists a pair of
$\pA$-invariant transitive measurable foliations $(\mathcal
F^{s},\mathcal F^{u})$ on a surface $S$ of genus $g$ that are
transverse to each other and have common singularities~$\Sigma_i$.
Furthermore, there must exist a constant $\lambda=\lambda(\pA)>1$ such
that $\pA$ expands leaves of one foliation and shrinks those of the
other foliation with coefficient $\lambda$ (in the sense of
measures). The number $\lambda$ is a topological invariant called the
\emph{dilatation} of $\pA$; the number $\log(\lambda)$ is the
\emph{topological entropy} of~$\pA$.

Thurston proved that if the foliations are orientable then 
\begin{enumerate}
\item The linear map $\pA_\ast$ defined on $H_1(S,\R)$ has a simple
  eigenvalue $\rho(\pA_\ast) \in \R$ such that $|\rho(\pA_\ast)| > |x|
  $ for all other eigenvalues $x$;
\item $|\rho(\pA_\ast)| > 1$ is the dilatation $\lambda$ of $\pA$. 
\end{enumerate}
In general $\lambda(\pA) \geq \rho(\pA_\ast)$, with equality if and
only if the foliations are orientable.

\begin{Remark}
  The number $|\rho(\pA_\ast)|$ is a \emph{Perron number} (i.e. an
  algebraic integer $\lambda > 1$ whose conjugates $\lambda'$ all
  satisfy $|\lambda'| < \lambda$).  If $\rho(\pA_\ast) > 0$ then
  $\rho(\pA_\ast)$ is a Perron root of the polynomial
  $\chi_{\pA_\ast}(\xx)$, where~$\chi_{\pA_\ast}$ is the
  characteristic polynomial of~$\pA_\ast$.  If $\rho(\pA_\ast) < 0$
  then $-\rho(\pA_\ast)$ is a Perron root of the polynomial
  $\chi_{\pA_\ast}(-\xx)$.  We also denote by~$\rho(P)$ the largest
  root (in magnitude) of the polynomial~$P$.
\end{Remark}

\begin{Definition}
  The isotopy class of a braid $\beta\in B_n$ is by definition the
  isotopy class of the corresponding homeomorphism $h$ of the
  $n$-punctured disc (relative to the set $\bigcup_i \Sigma_i
  \cup \partial D^2$, where $\Sigma_i$ are the singularities).  We say
  that the braid $\beta$ is pseudo-Anosov if the isotopy class of $h$
  is pseudo-Anosov; in this case we define the dilatation of $\beta$ as the
  dilatation of the pseudo-Anosov homeomorphism.
\end{Definition}

\subsection{Singularities}

Let $\Sigma$ be a singularity of the stable (or unstable) foliation
determined by $\pA$ such that there are $m$ leaves passing through
$\Sigma$ ($m\geq 1$, $m\not = 2$). We will say that $\Sigma$ is a
singularity of $\pA$ of degree $k=m-2$. We call \emph{separatrices}
leaves passing through singularities.  For an orientable foliation, if
$\Sigma \in S$ is a singularity of $\pA$ (necessarily of even degree,
say $2d$) then there are $2(d+1)$ emanating separatrices: $d+1$
outgoing separatrices and $d+1$ ingoing separatrices.  The collection
$s=(k_1,\dots,k_m)$ of the degrees of the singularities is called the
\emph{singularity data} of the foliations.  The Gauss-Bonnet formula
gives $\sum_{i=1}^m k_i = 4g-4$.

\subsection{Orientating double cover and strata}
\label{sec:strata}

The following construction is classical.  Let $h$ be a pseudo-Anosov
homeomorphism of the sphere (that fixes a marked point) and let
$(\mathcal F^s,\mathcal F^u)$ be the pair of {\it nonorientable}
invariant foliations on the sphere $\mathbb P^1$ determined by $h$.
There exists a canonical (ramified) double covering $\pi : S
\rightarrow \mathbb P^1$ such that the foliations lift to two
transverse measured foliations $(\widehat{\mathcal
  F^{s}},\widehat{\mathcal F^{u}})$ on $S$ that are orientable.
\medskip

\noindent The homeomorphism $h$ also lifts to a pseudo-Anosov
homeomorphism $\pA$ on $S$, with the same dilatation. Observe that if
we denote by $\tau$ the hyperelliptic involution determined by the
covering $\pi$, then there are two lifts: $\pA$ and $\tau \circ \pA$
($= \pA \circ \tau$).  We choose the lift $\pA$ such that
$\rho(\pA_\ast)>0$, so that the other lift satisfies $\rho((\tau \circ
\pA)_\ast)<0$.  The covering $\pi: S \rightarrow \mathbb P^1$ is the
minimal (ramified) covering such that the pullback of the foliations
$(\mathcal F^s,\mathcal F^u)$ becomes orientable.

The set of critical values of $\pi$ (i.e. the image of $\Fix(\tau)$ by
$\pi$) on the sphere coincides exactly with the set of singularities
of odd degree of the foliations.

\begin{Convention}
\label{convention:surface}
All surface homeomorphisms $\pA$ are defined on a surface $S$ equipped
with an involution $\tau$.  We will always assume that the
homeomorphisms $\pA$ and $\tau$ commutes (which is fulfilled if the
maps are affine with respect to the Euclidian metric determined by the
foliations~\cite{LanneauThiffeault2010}).
\end{Convention}

We can define strata for a pseudo-Anosov homeomorphism $\pA$ as
follows.  If $\pA$ fixes {\it globally} a set of $r$ singularities
(necessarily of the same degree $k$), we will use the superscript
notation $(k^r)$ for $k,\dots,k$ repeated $r$ times.  On the other
hand, if $\pA$ fixes {\it pointwise} the singularities, we will use
the notation $(k,\dots,k)$.  For instance, the singularity data
$(2^2,2,2)$ for $\pA$ on a genus-$3$ surface means that $\pA$ fixes a
set consisting of two degree-2 singularities and fixes the other two
degree-2 singularities pointwise. \medskip

For surfaces, we will allow {\it fake singularities}, i.e.\ regular
points, and we will use the formulation ``singularities of
degree~$0$''.  One has the following straightforward lemma:

\begin{NoNumberLemma}
  Let $h$ be a pseudo-Anosov homeomorphism on the sphere and let
  $\Sigma$ be a degree-$k$ singularity of $h$. Let $\pA$ be a lift of
  $h$ on the orientating double cover.  If $k$ is odd then $\Sigma$
  lifts to a single singularity of $\pA$ of degree $2k+2$; otherwise
  $\Sigma$ lifts to two singularities of $\pA$ of degree $k$.
\end{NoNumberLemma}

Convention~\ref{convention:sphere} prompts us to consider the following
two definitions.

\begin{Definition}
\label {def:stratum:sphere}
A stratum on the sphere is an unordered set of integers
$(k^{n_2}_2,\dots,k^{n_m}_m)$ with $\sum_{i=2}^m n_i k_i = -k_1-4$,
where $k_i \geq -1$ for any $i$. We will denote such a stratum by
$s=(k_1;k^{n_2}_2,\dots,k^{n_m}_m)$, the first element being special
--- it is the degree of the marked point of
Convention~\ref{convention:sphere}.
\end{Definition}

\begin{Definition}
\label {def:stratum:surface}
Let $S$ be a hyperelliptic surface (of genus $g\geq 1$) and let $\tau$
be the hyperelliptic involution. The singularity data of a
pseudo-Anosov $\pA$ determines a stratum on the surface
$$
( k^{n_1}_1, \dots, k^{n_l}_l , \underline{k_{l+1}^{n_{l+1}}},\dots ,\underline{k_m^{n_m}}), 
$$
where we underline the degree of the singularities that are permuted
by the involution $\tau$.
\end{Definition}

\begin{Example}
  Let $(-1;-1^{17},-1,4^2,2,1^2,3)$ be the singularity data of a
  pseudo-Anosov homeomorphism on the sphere. Then the corresponding
  singularity data for the lifts on the covering surface is
  $(0,0^{17},0,\underline{4}^4,\underline{2}^2,4^2,8)$; the genus of
  the covering surface is $10$.
\end{Example}

\subsection{Pseudo-Anosov homeomorphisms and the Lefschetz fixed point theorem}

We recall briefly the Lefschetz fixed point formula for
homeomorphisms on compact
surfaces~\cite{Brown_Lefschetz,LanneauThiffeault2010}.  In the present
section, let $\pA$ be a pseudo-Anosov homeomorphism of a compact
surface $S$ with \emph{orientable} invariant measured foliations.  If
$p$ is a fixed point of $\pA$, we define the index of $\pA$ at $p$ to
be the algebraic number $\textrm{Ind}(\pA,p)$ of turns of the vector
$(x,\pA(x))$ when $x$ describes a small loop around $p$.

\begin{Remark}
\label{rk:index}
The index at a fixed point is easy to calculate for a pseudo-Anosov
homeomorphism.  Let $\rho(\pA_\ast)$ be the leading eigenvalue of
$\pA_\ast$.  If $\rho(\pA_\ast) < 0$ then $\textrm{Ind}(\pA,p) = 1$
for any fixed point $p$ (regular or singular).  If $\rho(\pA_\ast) >
0$ then let $\Sigma$ be a fixed degree-$2d$ singularity of $\pA$
($d=0$ for a regular point).  It follows that either
\begin{itemize}
\item $\pA$ fixes each separatrix of $\Sigma$, hence
  $\textrm{Ind}(\pA,\Sigma) = 1-2(d+1) <0$, or
\item $\pA$ permutes cyclically the outgoing separatrices of $\Sigma$,
  hence $\textrm{Ind}(\pA,\Sigma) = 1$.
\end{itemize}
\end{Remark}

We will use the following corollary (see~\cite{LanneauThiffeault2010}
for a proof).
\begin{Corollary}
\label{cor:trick}
Let $\Sigma$ be a fixed degree-$2d$ singularity of $\pA$
with~$\rho(\pA_\ast) > 0$. If $\textrm{Ind}(\pA,\Sigma)=1$ then
$$
\forall \ 1 \leq i \leq d ,\ \textrm{Ind}(\pA^i,\Sigma)=1
$$
and 
$$
\textrm{Ind}(\pA^{d+1},\Sigma)=1-2(d+1).
$$
\end{Corollary}

\begin{NoNumberTheorem}[Lefschetz fixed point theorem]

  Let $\pA$ be a homeomorphism on a compact surface $S$.  Denote by
  $\Tr(\pA_\ast)$ the trace of the linear map $\pA_\ast$ defined on
  the first homology group $H_1(S,\R)$.  Then the Lefschetz number
  $L(\pA) = 2-\Tr(\pA_\ast)$ satisfies
\[
L(\pA) = \sum_{p=\pA(p)} \textrm{Ind}(\pA,p).
\]
\end{NoNumberTheorem}
\noindent For convenience, we isolate the key idea involved in the
Lefschetz formula since we will use it often.  If $\pA$ is a
pseudo-Anosov homeomorphism on a compact surface $S$, with
$\rho(\pA_{\ast})>0$, then
\begin{align*}
2-\Tr(\pA_\ast) &= \sum_\textrm{$p$ singular fixed points of $\pA$} \textrm{Ind}(\pA,p) - \#\{\textrm{regular fixed points of } \pA \}, \textrm{\quad and}\\
2+\Tr(\pA_\ast) &= \#\{\textrm{regular fixed points of } \tau\circ\pA \}.
\end{align*}
In particular
$$
2-\Tr(\pA_\ast) \leq \#\{\textrm{singular fixed points of } \pA \} - \#\{\textrm{regular fixed points of } \pA \}.
$$

\noindent A very useful proposition for calculating Lefschetz numbers
which we will use repeatedly without reference is the following.
\begin{Proposition}
\label{prop:calcul:trace}
Let $P$ be a degree-$2g$ monic reciprocal polynomial
$$
P=\xx^{2g} + \alpha \xx^{2g-1} + \beta \xx^{2g-2} + \gamma \xx^{2g-3} + \dots + \gamma \xx^3 + \beta\xx^2+\alpha\xx + 1.
$$
Let $\pA$ be a homeomorphism such that its characteristic polynomial
satisfies $\chi_{\pA_\ast} = P$. Then
$$
\begin{array}{l}
\textrm{Tr}(\pA_\ast) = -\alpha, \\
\textrm{Tr}(\pA^2_\ast) = \alpha^2 - 2\beta, \\
\textrm{Tr}(\pA^3_\ast) = -\alpha^3 + 3\alpha \beta - 3 \gamma.
\end{array}
$$
\end{Proposition}

\begin{proof}
Use Newton's formula, as in~\cite{LanneauThiffeault2010}.
\end{proof}

\subsection{Compatibility with the Lefschetz formula}

Let $P\in \Z[\xx]$ be a degree-$2g$ monic reciprocal polynomial. Let $s$ be a stratum of
the genus-$g$ surface $S$. Let us assume that there exists a pseudo-Anosov
homeomorphism $\pA$ on $S$ with $\chi_{\pA_\ast} = P$ and singularity data
$s$.  The traces of $\pA^\po_\ast$ (and so the Lefschetz numbers of iterates
of $\pA$) are easy to compute in terms of $P$. This gives algebraic
constraints on the number of periodic orbits of $\pA$ as well as the action of
$\pA$ on the separatrices.

\begin{Definition}
  Let $P$ be a degree-$2g$ monic reciprocal polynomial and $s$ be a
  stratum of the surface $S$. We will say that $P$ is compatible with
  $s$, or that $P$ is admissible, if there are no algebraic
  obstructions with the Lefschetz formula.
\end{Definition}

The following is clear:

\begin{Proposition}
\label{prop:admissible}
Let $h$ be a pseudo-Anosov homeomorphism on the sphere and let $\pA$
and $\tau \circ \pA$ be the two lifts on the covering surface $S$
(with singularity data $s$). If $P=\chi_{\pA_\ast}$ then the
polynomials $P(\xx)$ and $P(-\xx)$ are both compatible with $s$.
\end{Proposition}

We give two examples to illustrate the above proposition.

\subsubsection{First example}
The polynomial $P=\xx^2 - 3 \xx + 1$ is compatible with the stratum
$s=(0^4,\underline{0}^2)$.  Indeed, let us assume there exists an
Anosov homeomorphism $\pA$ on the torus with singularity data $s$.
The Lefschetz numbers of the first five iterates of $\pA$ are $(-1,
-5, -16, -45, -121)$.  We can check that the number of periodic orbits
of length $\po=(1,2,3,4,5)$ are $(0, 2, 5, 10, 24)$.  The same
calculation shows that $P(-\xx)$ is also compatible with $s$. Of
course an Anosov homeomorphism realizing $P$ does actually exists,
e.g. $\left[ \begin{smallmatrix}
2 & 1 \\ 1 & 1 \end{smallmatrix}  \right] \in \textrm{PSL}_{2}(\Z) \simeq \textrm{Mod}(1,0)$.

\subsubsection{Second example}
\label{ex:P2}
The polynomial $P= \xx^6 - 3 \xx^5 + 4 \xx^4 - 5 \xx^3 + 4 \xx^2 - 3
\xx + 1$ is compatible with $s=(0^5,0,4^2)$, but $P(-\xx)$ is not. Let
us show the latter assertion. Assume there exists $\pA$ with
$\chi_{\pA_\ast}(\xx)=P(-\xx)$. Since $\rho(P) > 0$, one has
$\rho(P(-\xx)) = \rho(\pA_\ast) < 0$; thus the index at each fixed
point (singular and regular) is $+1$ (see Remark~\ref{rk:index}). We
have $L(\pA) = 5$ so that $\pA$ has $5$ fixed points.  In particular
$\pA$ has at least $3$ regular fixed points. Now let $\psi=\pA^2$. Of
course $\rho(\psi_\ast) > 0$, so the index of $\psi$ at the
singularities is at most $1$, and the index of $\psi$ at the regular
points is $-1$. Since $\pA$ has at least three regular fixed points,
the same is true for $\psi$. The Lefschetz formula applied to $\psi$
gives $L(\psi) \leq 2 - 3 = -1$. But a straightforward calculation
produces $L(\psi) =1$ --- a contradiction.

\begin{Remark}
The converse of Proposition~\ref{prop:admissible} is not true in general. For 
instance the polynomials $P(\xx) = \xx^6 - 3 \xx^5 + 2 \xx^4 + 2 \xx^2 - 3 \xx + 1$ 
and $P(-\xx)$ are both compatible with the stratum $s= (0^5,0,4^2)$. Nevertheless 
there are no pseudo-Anosov homeomorphisms $\pA$ on a genus-$3$ surface with 
$\chi_{\pA_\ast}(\xx) = P(\xx)$ or $P(-\xx)$ (see~\cite{LanneauThiffeault_ttauto}).
\end{Remark}

We end this section by sketching a proof of our result.

\subsection{Outline of a proof of Theorem~\ref{theo:main}}
\label{sec:outline}

We outline the strategy for calculating the minimum dilatation
$\delta_n$ of pseudo-Anosov $n$-braids on the disc. Obviously
$\delta_n$ is the minimum of $\delta(s)$ where $s$ is taken over all
the strata on the sphere with $n$ punctures. Since we have a candidate
$\rho(P_{n})$ for $\delta_{n}$ (given by train track automata,
see~\cite{LanneauThiffeault_ttauto}) it is sufficient to show that
$\delta(s) \geq \rho(P_{n})$ for each stratum $s$.  The steps are as
follows:

\begin{enumerate}

\item Fix a stratum $s$ on the sphere with $n$ punctures.
Let $h$ be a pseudo-Anosov homeomorphism on the sphere
$\mathbb{P}^1$ that realizes $\delta(s)$. Let us denote by $\pi: S  \rightarrow \mathbb P^1$ 
the orientating double cover (genus($S)=g$). The homeomorphism $h$
lifts to two homeomorphisms $\pA$ and $\tau \circ \pA$ on $S$, where $\tau$
is the hyperelliptic involution of $S$. (By convention, we choose the lift
$\pA$ so that $\rho(\pA_\ast) > 0$.)  These two homeomorphisms also determine
a stratum $s'$ on $S$. Finally, there exists a reciprocal, monic, degree-$2g$ 
polynomial $P\in \Z[\xx]$ with a Perron
root $\delta(s)$ and, by Proposition~\ref{prop:admissible}, both $P(\xx)$ and
$P(-\xx)$ are compatible with the stratum $s'$.

\item We list all degree-$2g$ reciprocal monic polynomials $P$, with a
  Perron root $\rho(P)$, $1 < \rho(P) < \rho(P_{n})$ (there is a
  finite number of such polynomials~\cite{Arnoux1981,Ivanov1988}).  If
  $\delta(s) < \rho(P_{n})$ then $\delta(s)=\rho(P)$ for some other
  $P$ in our list.  We will rule out all such polynomials.

\item 
\label{step:list}
  Take a polynomial $P$ from our list.

\item
  If $P$ is not compatible with stratum $s'$ then go back to
  step~(\ref{step:list}) and move on to the next polynomial.
 
\item If $P(-\xx)$ is not compatible with the stratum $s'$ then go back to
  step~(\ref{step:list}) and move on to the next polynomial.

\item We exclude the remaining polynomials using combinatorics of the orbits
  or the action on the singularities.  
\end{enumerate}
    
\begin{Remark}
  One can actually obtain a more precise result: our techniques apply
  to (almost) all strata of the discs (see the appendix).
\end{Remark}

\section{Disc with $3$ punctures}
\label{section:n=3}

The only stratum on the sphere with four singularities is
$s=(-1;-1^3)$. The corresponding stratum on the associated orientating
double cover $S$ is $s'=(0,0^3)$.  Thus the genus of $S$ is~$1$. The
candidate for $\delta(s)$ is the Perron root of $\xx^2-3\xx+1$. A
direct calculation shows that there are no degree-2 monic reciprocal
polynomials with a Perron root strictly less than that of
$\xx^2-3\xx+1$. The statement is proved.

\section{Disc with $4$ punctures}
\label{section:n=4}

There are two strata to consider on the sphere: $s_1=(0; -1^4)$ and $s_2=(-1;
-1^4,1)$.  We check that the corresponding strata on the orientating double
cover $S_i$ are, respectively,
\[
  s'_1=(0^4,\underline{0}^2) \qquad \textrm{and} \qquad s'_2=(0,0^4,4).
\]
The candidate for $\delta_{4}$ is the Perron root of $\xx^4-2\xx^3-2\xx+1$.

\begin{itemize}
\item For the stratum $s_1$ the surface upstairs, $S_1$, has genus $1$ 
so again there is nothing to prove (see Section~\ref{section:n=3}).

\item For the stratum $s_2$ the surface upstairs, $S_2$, has
  genus $2$.  The next lemma shows that there are four degree-$4$
  monic reciprocal polynomials with a Perron root strictly less than
  our candidate.
\end{itemize}

\begin{Lemma}
\label{lm:monic:n=4}
The degree-$4$ monic reciprocal polynomials with a Perron root strictly less
than $\rho(\xx^4 - 2 \xx^3 - 2 \xx +1) \simeq 2.29663$ are
\begin{center}
\begin{tabular}{lc}
\hline
polynomial & Perron root \\
\hline
$\xx^4 - \xx^3 - \xx^2 - \xx + 1 $& 1.72208\\
$\xx^4 - 2 \xx^3 + \xx^2 - 2 \xx + 1 $& 1.88320\\
$\xx^4 - \xx^3 - 2 \xx^2 - \xx + 1$ & 2.08102\\
$\xx^4 - 3 \xx^3 + 3 \xx^2 - 3 \xx + 1 $ & 2.15372 \\
\hline
\end{tabular}
\end{center}
\end{Lemma}

\begin{proof}
Let $P=X^4+\alpha X^3 + \beta X^2 + \alpha X+1$ be a reciprocal Perron polynomial and let $Q$ be such 
that $X^4Q(X+X^{-1}) = P$. Observe that $\lambda$ is a root of $P$ if and only if $t=\lambda + \lambda^{-1}$ 
is a root of $Q=X^2 - a X+b$ where $\alpha = -a$ and $\beta = b + 2$. Since $t$ is an increasing function 
of $\lambda$, the polynomial $Q$ admits a Perron root $t$. Let $t'$ be its (real) Galois conjugate so that $|t'| < t$. \medskip

If $t' > 0$ then $a = t+t' >0$. If $t' < 0$ then $|t'| = -t' < t$ thus
$t+t'= a > 0$. Hence $a > 0$. Now $2.29663 + (2.29663)^{-1} =
2.73205$. If $a \geq 4$ then the Perron root of $Q$ satisfies
$$
2 + \sqrt{4-b} \leq t = \tfrac1{2}(a + \sqrt{a^2-4b}) < 2.73205.
$$
Thus $4-0.73205^2< b \leq 4$ and we derive $a=b=4$. Hence $t=2$ which
is a contradiction. Finally $a=1,2,$ or $3$. \medskip

If $a=1$ then $t=\frac1{2}(1 + \sqrt{1-4b}) < 2.73205$. Hence $b >
-4.73205$ so that $-4 \leq b \leq 0$. The cases $b=0,-1,-2$ lead to
non-Perron roots and the cases $b=-3,-4$ lead to the two polynomials
$Q=\xx^2-\xx-3$ and $Q=\xx^2-\xx-4$, that is, the polynomials $P=
\xx^4 - \xx^3 - \xx^2 - \xx + 1$ and $P= \xx^4 - \xx^3 - 2 \xx^2 - \xx
+ 1 $ respectively. \medskip

The second case $a=2$ leads to $t=1 + \sqrt{1-b} < 2.73205$. Hence $b
> -2$ and $b=-1$ or $b=0$.  The case $b=0$ leads to non-Perron roots
and the case $b=-1$ leads to the polynomial $Q=\xx^2-2\xx-1$, that is,
the polynomial $P= \xx^4 - 2\xx^3 + \xx^2 -2 \xx + 1$. \medskip

The third case $a=3$ leads to $t=\frac1{2}(3 + \sqrt{9-4b}) <
2.73205$. Hence $b > 0.732051$, implying $b=1$. The corresponding
polynomial is $Q=\xx^2-3\xx+1$, that is, the polynomial $P= \xx^4 -
3\xx^3 + 3\xx^2 -3 \xx + 1$.
\end{proof}

The next lemma will rule out the first and the third polynomials.

\begin{Lemma}
Let $\pA$ be a pseudo-Anosov homeomorphism on $S_2$ with singularity data
$(4,0)$. If $\rho(\pA_\ast)>0$, then
$$
\textrm{Tr}(\pA_\ast) \ge 2.
$$
\end{Lemma}

\begin{proof}
If $\Sigma\in S_2$ is the degree-four singularity then one has
$$
2-\textrm{Tr}(\pA_\ast) = \textrm{Ind}(\pA,\Sigma) - \#\Fix(\pA),
$$
where $\Fix(\pA)$ is the set of {\it regular} fixed points of $\pA$.
Since $\#\Fix(\pA) \geq 1$ and $\textrm{Ind}(\pA,\Sigma) \leq 1$, we get the desired 
inequality.
\end{proof}

This rules out the first and the third polynomials since for those we
would have $\textrm{Tr}(\pA_\ast) = 1$.  Now let us show that the
second polynomial is also inadmissible. \medskip

Assume the second polynomial is admissible and let $\pA$ be a
pseudo-Anosov homeomorphism such that $\chi_{\pA_\ast}(\xx)=\xx^4 - 2
\xx^3 + \xx^2 - 2 \xx + 1$. Then $L(\pA)=0$. Since the singularity
data of $\pA$ is $(4,0)$, $\pA$ fixes the degree-four singularity
(with positive index) and has only one regular fixed point (which is
also fixed by $\tau$ by construction).  Since $L(\pA^2)=0$ the same
argument applies to $\pA^{2}$ so that $\pA$ has no period-2 orbits.

Now let us count the number of fixed points of $\tau \circ \pA$. These
are fixed points of $\pA^{2}$. Indeed if $\tau \circ \pA(p)=p$ then
$\pA(p)=\tau(p)$ and $\pA^2(p)= \tau \circ \pA(p)=p$. By the above
discussion, $\tau \circ \pA$ has only one regular fixed point (hence
also fixed by $\tau$). Thus $\tau \circ \pA$ has two fixed points in
total (one singular and one regular). But
\[
  L(\tau \circ \pA) = 2 - \textrm{Tr}((\tau \circ \pA)_\ast) = 2 +
  \textrm{Tr}(\pA_*) = 4.
\]
Since $L(\tau \circ \pA)$ is the number of fixed points (singular and regular, see Remark~\ref{rk:index}) of 
$\tau \circ \pA$, we get a contradiction. \medskip

Finally, to finish the proof of the $n=4$ case let us show that the
fourth polynomial is inadmissible.  Assume it is admissible, which
implies~$L(\pA)=-1$. Then $\pA$ must fix the singularity with positive
index (since otherwise $L(\pA) \leq -5$) and has two regular fixed
points. Since $L(\pA^2)=-1$ as well, the same argument shows that
$\pA^{2}$ has
two regular fixed points (the same as $\pA$) and so $\pA$ has no period-$2$ orbits. \\
As we have seen, fixed points of $\tau \circ \pA$ are also fixed
points of $\pA^{2}$; thus $\tau \circ \pA$ has three fixed points in
total (one singular and two regular). But
\[
  3 = L(\tau \circ \pA) = 2 - \textrm{Tr}((\tau \circ \pA)_*) = 2 +
  \textrm{Tr}(\pA_*) = 5,
\]
which is a contradiction. Hence $\delta(s_2) = \rho(\xx^4 - 2 \xx^3 - 2 \xx^2 + 1)$ and 
Theorem~\ref{theo:main} for $n=4$ is proved.

\section{Disc with $5$ punctures}
\label{section:n=5}

There are four strata to consider on the sphere: 
$$
s_1=(1; -1^5), \qquad s_2=(0;-1^5,1), \qquad s_3=(-1;-1^5,2) \qquad 
\textrm{and} \qquad s_4=(-1;-1^5,1^2).
$$
The corresponding strata on the orientating double cover $S_i$ are
\[
  s'_1=(0^5,4), \qquad s'_2=(0^5,4,\underline{0}^2), \qquad
  s'_3=(0^5,0,\underline{2}^2) \quad \textrm{and} \quad s'_4=(0^5,0,4^2).
\]
The candidate for $\delta_{5}$ is the Perron root of
$\xx^4-\xx^3-\xx^2-\xx+1$.  For the first three cases, there is
nothing to prove since by Lemma~\ref{lm:monic:n=4} there are no
degree-$4$ monic reciprocal polynomials having a Perron root less than
$\delta_{5}$.  We use a computer for the last case
(analogous to Lemma~\ref{lm:monic:n=4}).  This is
straightforward. \medskip

The surface $S_4$ has genus $3$. There are~$9$ degree-$6$ monic reciprocal polynomials with a Perron root
strictly less than our candidate, namely:
\begin{center}
\begin{tabular}{lc}
\hline
polynomial & Perron root \\
\hline
$\xx^{6} + \xx^{5} - \xx^4 -  3 \xx^3 - \xx^2 + \xx + 1$		& $1.32472 $ \\
$\xx^{6} - \xx^4 - \xx^3 - \xx^2 + 1$  & $1.40127 $ \\
$\xx^{6} - \xx^{5} + \xx^4 -  3 \xx^3 + \xx^2 - \xx + 1$  & $1.46557 $ \\
$\xx^{6} - \xx^{5} - \xx^3 - \xx + 1$  & $1.50614 $ \\
$\xx^{6} - \xx^{5} - \xx^4 + \xx^3 - \xx^2 - \xx + 1$  & $1.55603 $ \\
$ \xx^{6} - 2 \xx^{5} + 3 \xx^4 - 5 \xx^3 + 3 \xx^2 -  2 \xx + 1$  & $1.56769 $ \\
$\xx^{6} - \xx^4 - 2 \xx^3 - \xx^2 + 1$  & $1.58235 $ \\
$ \xx^{6} - 2 \xx^{5} + 2 \xx^4 - 3 \xx^3 +  2 \xx^2 - 2 \xx + 1$  & $1.63557 $ \\
$\xx^{6} - \xx^{5} + \xx^4 - 4 \xx^3 + \xx^2 - \xx + 1$  & $1.67114$ \\
\hline
\end{tabular}
\end{center}

\noindent We will use a simple algebraic criterion in order to
eliminate these polynomials.

\begin{Lemma}
\label{lm:eliminate4}
Let $\pA$ be a pseudo-Anosov homeomorphism on $S_4$ with singularity data
$(4,4,0)$, with $\rho(\pA_\ast)>0$.  Then $L(\pA^\po) \leq 1$ for any $\po\geq
1$. In addition, if $L(\pA)\geq 0$ then
\[
  L(\pA^3) \leq -11.
\]
\end{Lemma}

\begin{proof}[Proof of the lemma]
The first remark $L(\pA^\po) \leq 1$ is trivial: positive indices arise only
from singularities, so that $L(\pA^\po) \leq 2$. But since $\pA$ fixes a
regular point for any $\po$ the first statement holds. \medskip

Now let us assume in addition that $L(\pA) \geq 0$.  Since $\pA$ fixes
a regular point, if $\pA$ permutes the two singularities then $L(\pA)
<0$, a contradiction.  Hence, $\pA$ fixes the two singularities, with
at least one of the singularities having positive index.  Assume that
the other singularity is fixed with \emph{negative} index. Then
$L(\pA) \leq 1 - 5 - \#\textrm{\{Regular fixed points\}} < 0$ which is
again a contradiction.  Hence the index at the two singularities is
positive. Thus the third power of $\pA$ fixes the singularities and
their separatrices, and
\[
  L(\pA^3) = -5 -5 - \#\textrm{\{Regular fixed points of $\pA^3$\}} \leq -11
\]
which proves the lemma.
\end{proof}

The above lemma can be restated as follows:

\begin{Lemma}
Let $P=\xx^6+\alpha \xx^5 +\beta \xx^4+\gamma \xx^3 + \beta \xx^2 + \alpha \xx + 1$ be a
degree-$6$ monic reciprocal polynomial. If there exists a pseudo-Anosov homeomorphism $\pA$ with 
$\rho(\pA_\ast)>0$, $\chi_{\pA_\ast} = P$, and with singularity data $(4^2,0)$, then
$2+\alpha \leq 1$ and $2-\alpha^2+2\beta \leq 1$. \medskip
Moreover if $2+\alpha \geq 0$ then $2+\alpha^3-3\alpha \beta + 3 \gamma \leq -11$.
\end{Lemma}

None of the~$9$ polynomials above satisfies this algebraic criterion;
thus Theorem~\ref{theo:main} for $n=5$ is proved.
%
%
%

\section{Disc with $6$ punctures}
\label{section:n=6}

The techniques of the previous sections can also be applied to the
case $n=6$. However the complexity becomes huge so we rely on a set of
\Mathematica~\cite{Mathematica7} functions to test whether a
polynomial~$P$ is compatible with a given stratum.  This is
straightforward: we simply try all possible permutations of the
singularities and separatrices (there is a finite number of these),
and calculate the contribution to the Lefschetz numbers for each
iterate of~$\pA$.  We then check whether the deficit in the Lefschetz
numbers can be exactly compensated by regular periodic orbits.  If
not, the polynomial~$P$ cannot correspond to a characteristic
polynomial of some pseudo-Anosov homeomorphism on that stratum.  If it
can, we also test the polynomial~$P(-X)$ corresponding
to~$\rho(\pA_\ast) < 0$.

\subsection{Puncturing a singularity of higher degree}

So far we have considered strata on the sphere where only
singularities of degree $-1$ or the marked singularity (corresponding
to the disc's boundary) are punctured.  In general, it is possible to
have higher-degree punctured singularities.  This does not yield any
new pseudo-Anosov homeomorphisms: for instance, taking the stratum
$(0;-1^6,2)$ of the sphere (corresponding to a stratum on the disc
with $6$ punctures) and puncturing the degree-$2$ singularity gives a
stratum of the sphere corresponding to a stratum on the disc with $7$
punctures.  However, the pseudo-Anosov homeomorphisms on this new
stratum are identical to the original since the degree-$2$ singularity
must be fixed anyways.

Another example is to puncture the degree-$1$ singularity of the
stratum $(0;-1^5,1)$ (arising from a stratum of the disc with 5
punctures; see~\cite{Venzke_thesis}). This produces a braids on the
disc with~$6$ punctures with dilatation~$1.72208$; namely
$\sigma_2\sigma_1\sigma_2\sigma_1\left( \sigma_1 \sigma_2 \sigma_3
  \sigma_4 \sigma_5 \right)^2$.  We will show in the next section that
$\delta_6 \simeq 1.72208$.  For this reason, the pseudo-Anosov braid
with the least dilatation on the disc with $6$ punctures actually
arises from a stratum of the disc with $5$ punctures. 

Puncturing a degree-$1$ singularity in the stratum $(-1;-1^5,1,1)$
does not give new pseudo-Anosov homeomorphisms, since this merely
eliminates the ones that permute the two degree-$1$ singularities.
However, the least dilatation on the stratum $(-1;-1^5,1,1)$ with a
punctured degree-$1$ singularity can be larger than the unpunctured
case.  We do not consider such cases here.
See~\cite{Ham2007,Venzke_thesis} for more details.

Thus we reduce Theorem~\ref{theo:main}, case $n=6$ to the following:

\begin{Theorem}
  Let $s_i$ be a stratum of the sphere where the punctures are only
  singularities of degree one or regular points.  Then $\delta(s_i) >
  \rho(\xx^4-\xx^3-\xx^2-\xx-1) \simeq 1.72208$.
\end{Theorem}

\subsection{Strata of the six punctured disc}

In Table~\ref{strata:n=6} we present the list of all possible strata
of the six-punctured disc with the corresponding strata on the surface
$S$, and the list of polynomials with Perron root strictly less
than~$\rho(\xx^4-\xx^3-\xx^2-\xx-1)$, that of the candidate polynomial.
Under ``\# polynomials'' we give the number of polynomials $P$ with Perron
root strictly less than our candidate, and under ``\# compatible'' we
give how many of these are both compatible ($P(X)$ and $P(-X)$) with the Lefschetz formula for
that stratum. 

\begin{table}[htbp]
\begin{tabular}{cllccc}
\hline
case & stratum on $\mathbb P^1$ & stratum on $S$  & genus of $S$ & \# polynomials & \# compatible \\
\hline
$s'_1$ & $(2;-1^6)$ & $(0^6,\underline{2}^2)$ & 2 & 0 & 0 \\
$s'_2$ & $(1;-1^6,1)$ & $(0^6,4,4) $ & 3 & 9 & 0 \\
$s'_3$ & $(0;-1^6,2)$ & $(0^6,\underline{0}^2,\underline{2}^2)$ & 2 & 0 & 0 \\
$s'_4$ & $(-1;-1^6,3)$ & $(0^6,0,8) $ & 3 & 9 & 0 \\
$s'_5$ & $(0;-1^6,1^2)$ & $(0^6,4^2,\underline{0}^2)$ & 3 & 9 & 0 \\
$s'_6$ & $(-1;-1^6,1,2)$ & $(0^6,0,4,\underline{2}^2)$ & 3 & 9 & 0 \\
$s'_7$ & $(-1;-1^6,1^3)$ & $(0^6,0,4^3)$ & 4 & 148 & 2 \\
\hline
\end{tabular}
\medskip
\caption{Strata of the six-punctured disc.}
\label{strata:n=6}
\end{table}

\subsubsection{The cases $s_2$, $s_4$, $s_5$ and $s_6$}

We use a \Mathematica~\cite{Mathematica7} program to eliminate the
polynomials using the Lefschetz formula, and find that there are no
such polynomials.

\subsubsection{The case $s_7$}

Since the number of polynomials is large, we use a
\Mathematica~\cite{Mathematica7} program to eliminate the polynomials
using the Lefschetz formula.  Among the $148$ polynomials, there are
only two polynomials that satisfy this criterion, namely
$$
\begin{tabular}{lc}
\hline
polynomial & Perron root \\
\hline
$\xx^8 - 2 \xx^7 + 2 \xx^6 - 4 \xx^5 + 5 \xx^4 - 4 \xx^3 + 2 \xx^2 - 2 \xx + 1$		& 1.59937 \\
$\xx^8 - 3 \xx^7 + 4 \xx^6 - 7 \xx^5 + 10 \xx^4 - 7 \xx^3 + 4 \xx^2 - 3 \xx + 1$	& 1.67114 \\
\hline
\end{tabular}
$$
We will rule out these two polynomials by considering the $\pA$-action
on the singularities. \bigskip

For the first polynomial, assume there exists $\pA$ such that
$\chi_{\pA_\ast} =P$. We first show that $\pA$ fixes only one degree-4
singularity. Recall that $\pA$ fixes (at least) one regular point.  If
$\pA$ permutes the singularities then $L(\pA) \leq -1$, contradicting
$L(\pA) = 0$. If $\pA$ fixes pointwise the singularities (necessarily
with positive index) then it has three regular fixed points.  Hence
$\pA^2$ has at least three regular fixed points and still fixes
pointwise the singularities (with positive index). Thus
$$
L(\pA^2) \leq 3 - 3 = 0 
$$
contradicting $L(\pA^2) = 2$. We have proved that $\pA$ fixes only one singularity and so the same is true 
for $\tau \circ \pA$. Since $L(\tau \circ \pA) = 4$ one can conclude that $\pA$ has $3$ regular fixed points. This implies 
that $L(\pA) \leq 1 - 3 < 0$ contradicting again $L(\pA) = 0$.

For the second polynomial, since $L(\pA) = 1$ and $L(\pA^{2}) =0$ the
same argument shows that $\pA$ fixes only one degree-4 singularity, so
the same is true for $\tau \circ \pA$.  Since $L(\tau \circ \pA) = 5$
one can conclude that $\tau \circ \pA$ has $4$ regular fixed
points. In particular $\pA$ has at least two regular fixed points and
so $L(\pA) \leq -1$, which is not possible.

Theorem~\ref{theo:main} for $n=6$ is proved.

\section{Disc with $7$ punctures}
\label{section:n=7}

Again the techniques of the previous sections can be applied to the
seven-punctured disc. Table~\ref{strata:n=7} gives the list of all
possible strata of the seven-punctured disc with the corresponding
strata on the surface $S$, and the number of polynomials with a
dilatation smaller than the candidate polynomial
($\rho(\xx^7-2\xx^4-2\xx^3+1)=\delta_7\simeq1.46557$).  The only
stratum we have to check is~$s_{12}$. The two compatible polynomials
are
$$
\begin{tabular}{lc}
\hline
polynomial & Perron root \\
\hline
$\xx^{10} - 4 \xx^{9} + 5 \xx^8 - \xx^7 - 2 \xx^6 + 2 \xx^5 - 2 x^4 - x^3 + 5 x^2 -  4 x + 1$ & 1.40127 \\
$\xx^{10} - 2 \xx^{9} + x^7 + x^6 - 2 x^5 + x^4 + x^3 - 2 x + 1$    & 1.45799 \\
\hline
\end{tabular}
$$
There are no pseudo-Anosov homeomorphisms that realize these
polynomials. Indeed, if we assume there is one, then the same
arguments as in the previous section give a contradiction on the
number of periodic points (periodic orbits of length $14$ and of
length $7$, respectively).  We leave the details to the reader.
\begin{table}[htbp]
\begin{tabular}{cllccc}
\hline
case & stratum on $\mathbb P^1$ & stratum on $S$ & genus of $S$ & \# polynomials & \# compatible \\
\hline
$s'_1$ & $(3;-1^7)$ & $(0^7,8)$ & 3 & 2 & 0 \\
$s'_2$ & $(2;-1^7,1)$ & $(0^7,4,\underline{2}^2)$ & 3 & 2 & 0 \\
$s'_3$ & $(1;-1^7,2)$ & $(0^7,4,\underline{2}^2)$ & 3 & 2 & 0 \\
$s'_4$ & $(0;-1^7,3)$ & $(0^7,8,\underline{0}^2)$ & 3 & 2 & 0 \\
$s'_5$ & $(-1;-1^7,4)$ & $(0^7,0,\underline{4}^2)$ & 3 & 2 & 0 \\
$s'_6$ & $(1;-1^7,1^2)$ & $(0^7,4^2,4)$ & 4 & 21 & 0 \\
$s'_7$ & $(0;-1^7,1,2)$ & $(0^7,4,\underline{0}^2,\underline{2}^2)$ & 3 & 2 & 0 \\
$s'_8$ & $(-1;-1^7,1,3)$ & $(0^7,0,4,8)$ & 4 & 21 & 0 \\
$s'_9$ & $(-1;-1^7,2^2)$ & $(0^7,0,\underline{2}^4)$ & 3 & 2 & 0 \\
$s'_{10}$ & $(0;-1^7,1^3)$ & $(0^7,0,4^3)$ & 4 & 21 & 0 \\
$s'_{11}$ & $(-1;-1^7,1^2,2)$ & $(0^7,0,4^2,\underline{2}^2)$ & 4 & 21 & 0 \\
$s'_{12}$ & $(-1;-1^7,1^4)$ & $(0^7,0,4^4)$ & 5 & 227 & 2 \\
\hline
\end{tabular}
\medskip
\caption{Strata of the seven-punctured disc.}
\label{strata:n=7}
\end{table}

\section{Disc with $8$ punctures}
\label{section:n=8}

Again the techniques of the previous sections can be applied to the
eight-punctured disc. Table~\ref{strata:n=8} gives the list of all
strata of the eight-punctured disc with the corresponding strata on
the surface $S$, and the number of polynomials with a dilatation
smaller than the candidate polynomial
($\rho(\xx^8-2\xx^5-2\xx^3+1)=\delta_8\simeq1.41345$), as given
by~\cite{Venzke_thesis}.  There are five polynomials to check, and
these are readily eliminated by examining the detailed action on
singularities.  The method is identical to the previous sections so we
omit the detailed argument here.
\begin{table}[htbp]
\begin{tabular}{cllccc}
\hline
case & stratum on $\mathbb P^1$ & stratum on $S$ & genus of $S$ & \# polynomials & \# compatible \\
\hline
$s'_1$ & $(4;-1^8)$ & $(0^8,\underline{4}^{2})$ & 3 & 2  & 0  \\
$s'_2$ & $(3;-1^8,1)$ & $(0^8,4,8)$ & 4 & 15 & 0 \\
$s'_3$ & $(2;-1^8,2)$ & $(0^8,\underline{2}^{2},\underline{2}^{2})$ & 3 & 2 & 0 \\
$s'_4$ & $(1;-1^8,3)$ & $(0^8,4,8)$ & 4 & 15 & 0 \\
$s'_5$ & $(0;-1^8,4)$ & $(0^8,\underline{0}^{2},\underline{4}^{2})$ & 3 & 2 & 0  \\
$s'_6$ & $(-1;-1^8,5)$ & $(0^8,0,12)$ & 4 & 15 & 1 \\
$s'_7$ & $(2;-1^8,1^{2})$ & $(0^8,\underline{2}^{2},4^{2})$ & 4 & 15 & 0  \\
$s'_8$ & $(1;-1^8,1,2)$ & $(0^8,4,4,\underline{2}^{2})$ & 4 & 15 & 0 \\
$s'_9$ & $(0;-1^8,1,3)$ & $(0^8,\underline{0}^{2},4,8)$ & 4 & 15 & 0 \\
$s'_{10}$ & $(-1;-1^8,1,4)$ & $(0^8,0,4,\underline{4}^{2})$ & 4 &15  & 0 \\
$s'_{11}$ & $(0;-1^8,2^{2})$ & $(0^8,\underline{0}^{2},\underline{2}^{4})$ & 3 & 2 & 0 \\
$s'_{12}$ & $(-1;-1^8,2,3)$ & $(0^8,0,\underline{2}^{2},8)$ & 4 &15  & 0 \\
$s'_{13}$ & $(1;-1^8,1^{3})$ & $(0^8,4,4^{3})$ & 5 & 129  & 2 \\
$s'_{14}$ & $(0;-1^8,1^{2},2)$ & $(0^8,\underline{0}^{2},4^{2},\underline{2}^{2})$ & 4 &15  & 0  \\
$s'_{15}$ & $(-1;-1^8,1^{2},3)$ & $(0^8,0,4^{2},8)$ & 5 & 129 & 0 \\
$s'_{16}$ & $(-1;-1^8,1,2^{2})$ & $(0^8,0,4,\underline{2}^{4})$ & 4 & 15 & 0 \\
$s'_{17}$ & $(0;-1^8,1^{4})$ & $(0^8,\underline{0}^{2},4^{4})$ & 5 & 129 & 2 \\
$s'_{18}$ & $(-1;-1^8,1^{3},2)$ & $(0^8,0,4^{3},\underline{2}^{2})$ & 5 & 129 & 0 \\
$s'_{19}$ & $(-1;-1^8,1^{5})$ & $(0^8,0,4^{5})$ & 6 & 1096 & 0 \\
\hline
\end{tabular}
\medskip
\caption{Strata of the eight-punctured disc.}
\label{strata:n=8}
\end{table}

\appendix

\section{Minimum dilatation for each stratum}

In this appendix, we give the minimum dilatation for each stratum and
give explicit examples of a braid realizing each minimum, for $3\le n
\le 7$. We first detail the $n=5$ case (Section~\ref{appendix:n5}) and
then give the other cases without proof since the techniques are the
same (Section~\ref{appendix:general}).

\subsection{Strata for the disc with $5$ punctures}
\label{appendix:n5}

Here we give an alternative proof to that of Ham \&
Song~\cite{Ham2007}.  There are four strata to consider on the sphere:
$$
s_1=(1; -1^5), \qquad s_2=(0;-1^5,1), \qquad s_3=(-1;-1^5,2) \qquad 
\textrm{and} \qquad s_4=(-1;-1^5,1^2).
$$
The corresponding strata on the orientating double cover $S_i$ are
\[
  s'_1=(0^5,4), \qquad s'_2=(0^5,4,\underline{0}^2), \qquad
  s'_3=(0^5,0,\underline{2}^2) \quad \textrm{and} \quad s'_4=(0^5,0,4^2).
\]
The candidate dilatations for each stratum are given in
Table~\ref{tab:disc5}.  We will examine each stratum in turn.

\subsubsection{The strata $s_1$ and $s_2$} The candidate dilatation
for these strata is the Perron root of
$\xx^4-\xx^3-\xx^2-\xx+1$. Lemma~\ref{lm:monic:n=4} implies there are
no degree-4 monic reciprocal polynomials with a Perron root strictly
less than that of the candidate.

\subsubsection{The stratum $s_3$} The genus of $S_3$ is $2$. Again Lemma~\ref{lm:monic:n=4} implies 
there are three degree-4 monic reciprocal polynomials with a Perron root strictly less than our candidate
for $\delta(s_3)$. The next lemma will rule out these three polynomials.

\begin{Lemma}
\label{lm:eliminate3}
Let $\pA$ be a pseudo-Anosov homeomorphism on $S$ with singularity data 
$(2,2,0)$. Assume in addition that $\rho(\pA_\ast)>0$. If $\textrm{Tr}(\pA_\ast) \leq 2$ then
\[
  \textrm{Tr}(\pA^2_\ast) \geq 9.
\]
\end{Lemma}

We first show how to rule out the polynomials and then we will prove
the lemma. For each polynomial $P$, let us assume there exists $\pA$
with $\chi_{\pA_\ast}=P$; then we can see that $\textrm{Tr}(\pA_\ast)
\geq 2$, so that we should have $\textrm{Tr}(\pA^2_\ast) \geq 9$. But
Proposition~\ref{prop:calcul:trace} shows that $\textrm{Tr}(\pA^2_\ast) \leq
5$. We now prove the lemma.

\begin{proof}[Proof of Lemma~\ref{lm:eliminate3}]
  First of all observe that $\textrm{Tr}(\pA_\ast) \leq 2$ if and only
  if $L(\pA) \geq 0$, and $\textrm{Tr}(\pA^2_\ast) \geq 9$ if and only
  if $L(\pA^2) \leq -7$. Thus let us assume $L(\pA) \geq 0$.  Since
  $\pA$ fixes a regular point, $\pA$ fixes the two degree-2
  singularities with positive index (otherwise we would have $L(\pA)
  <0$).  Hence $\pA^2$ fixes the two singularities and their separatrices, so
  that the index of $\pA^2$ at the singularities is $-3$. Now
\[
  L(\pA^2) = -3 -3 - \#\Fix(\pA^2) \leq -7,
\]
where $\Fix(\pA^2)$ is the set of regular fixed points of $\pA^2$. The
lemma is proved.
\end{proof}

\subsubsection{The stratum $s_4$}\label{sec:s4dic5} The surface $S_4$ has genus $3$.
There are~$41$ degree-$6$ monic reciprocal polynomials with a Perron
root strictly less than our candidate $\delta(s_4)$. Among these
polynomials, there are only $3$ that satisfy the conclusion of Lemma~\ref{lm:eliminate4}, namely
\begin{align*}
  P_1 &= \xx^6 - 3 \xx^5 + 2 \xx^4 + 2 \xx^2 - 3 \xx + 1, \\
  P_2 &= \xx^6 - 3 \xx^5 + 4 \xx^4 - 5 \xx^3 + 4 \xx^2 - 3 \xx + 1,
  \quad\textrm{ and } \\
  P_3 &= \xx^6 - 4 \xx^5 + 6 \xx^4 - 6 \xx^3 + 6 \xx^2 - 4 \xx + 1.
\end{align*}
We discuss now the compatibility of the three polynomials $P_1$,
$P_2$, and $P_3$.  \medskip

\noindent First of all $P_2(-\xx)$ is not compatible with the stratum
$s'_4$ (see Example~\ref {ex:P2}). \medskip

For the third polynomial, both $P_3(\xx)$ and $P_3(-\xx)$ are
admissible, which means the Lefschetz formula does not rule out
$P_3$. Let us assume that there exists $\pA$, with singularity data
$s'_4$, such that $\chi_{\pA_\ast} = P_3$.  We will get a
contradiction using the $\pA$-action on the singularities. More
precisely we will show $\pA$ permutes the singularities whereas $\tau
\circ \pA$ fixes them, which is impossible (since $\tau$ fixes the
singularities).

First assume that the two singularities are fixed by $\pA$. Since
$L(\pA)=-2$ the index at these singularities is positive, otherwise
$L(\pA) \leq -5 + 1 = -4$. Now $\pA^3$ fixes the singularities with
negative index, so that $L(\pA^3) \leq -5 -5 = -10$. This is a
contradiction with $L(\pA^3)=-8$. Thus $\pA$ has to permutes the two
singularities. Now let us show that $\tau \circ \pA$ fixes the two
singularities.  If not, then since $L(\tau \circ \pA) = 6$ the
homeomorphism $\tau \circ \pA$ has $6$ regular fixed points, so that
$(\tau \circ \pA)^2=\pA^2$ also has $6$ regular fixed points. Hence
$L(\pA^2) \leq 2 - 6 = -4$ but $L(\pA^2)=1$: we have a
contradiction. \medskip

Finally, the first polynomial $P_1$ cannot be ruled out using the
Lefschetz formula. Indeed it is compatible with the stratum $s'_4$ and
there are no contradictions with the $\pA$-action on the
singularities.  Thus the Lefschetz formula does not help to calculate
$\delta(s_{4})$ and we need a more elaborate argument to conclude.
Using train track automata, one can actually prove
$\delta(s_4)\simeq2.01536$~\cite{Ham2007,LanneauThiffeault_ttauto}.

\subsection{Minimum dilatation for each stratum}
\label{appendix:general}

We give the minimum dilatation for each stratum and provide explicit
examples of a braid realizing each minimum. We have indicated by a
star the cases where we cannot conclude using the Lefschetz formula
only. For these cases one can conclude using train track automata
(see~\cite{Ham2007} for $n=5$ and~\cite{LanneauThiffeault_ttauto} for
$n=6,7$).

\medskip

\noindent Note that in the tables we denote~$\Delta_k$ the braid
$\sigma_1\cdots\sigma_k$.  We used Toby Hall's implementation of the
Bestvina--Handel algorithm~\cite{HallTrain,Bestvina1995} to verify
that the braids correspond to pseudo-Anosov homeomorphisms.

\begin{table}[htbp]
\begin{tabular}{lll}
\begin{tabular}{llll}
\hline
case & $\delta(s) \simeq$ & polynomial & braid \\
\hline
$s_1$ & 2.61803 & $\xx^2-3\xx+1$ & $\sigma_1\sigma_2^{-1}$ \\
\hline
\end{tabular}& \qquad&
\begin{tabular}{llll}
\hline
case & $\delta(s) \simeq$ & polynomial & braid \\
\hline
$s_1$& 2.61803 & $\xx^3-2\xx^2-2\xx+1$ &
$\sigma_1\sigma_2\sigma_1\sigma_2\sigma_3^{-1}\Delta_3$\\
$s_2$& 2.29663 & $\xx^4-2\xx^3-2\xx+1$ & $\sigma_1\sigma_2\sigma_3^{-1}$\\
\hline
\end{tabular}
\end{tabular}
\medskip
\caption{Discs with $3$ and $4$ punctures.}
\end{table}

\begin{table}[htbp]
\begin{tabular}{llll}
\hline
case & $\delta(s) \simeq$ & polynomial & braid \\
\hline
$s_1$& 1.72208 & $\xx^4-\xx^3-\xx^2-\xx+1$ & $\deltthree\deltfour \sigma_3^{-1}$ \\
$s_2$& 1.72208 & $\xx^5-2\xx^3-2\xx^2+1$ & $\sigma_1^2\deltfour^2$ \\
$s_3$& 2.15372 & $\xx^5-2\xx^4-2\xx+1$ & $\deltthree\sigma_4^{-1}$ \\
$s_4 \ *$& 2.01536 & $\xx^6-\xx^5-4\xx^3-\xx+1$ & $\sigma_1\sigma_2\sigma_4^{-1}\sigma_3^{-1}$ \\
\hline
\end{tabular}
\medskip
\caption{Disc with $5$ punctures.}
\label{tab:disc5}
\end{table}

\begin{table}[htbp]
\begin{tabular}{llll}
\hline
case & $\delta(s)$ & polynomial & braid \\
\hline
$s_1$& 1.88320 & $\xx^5-\xx^4-\xx^3-\xx^2-\xx+1$ & $\Delta_5\sigma_4\sigma_5$ \\
$s_2$& 1.83929 & $\xx^6-\xx^4-4\xx^3-\xx^2+1$ & $\sigma_5\sigma_4^{-1}\deltfive^2$ \\
$s_3$& 1.88320 & $\xx^6-2\xx^4-2\xx^3-2\xx^2+1$ & $\sigma_1^2\sigma_4\deltfive^2$ \\
$s_4 \ *$& 2.08102 & $\xx^6-2\xx^5-2\xx+1$ & $\Delta_4\sigma_5^{-1}$ \\
$s_5 \ *$& 2.08102 & $\xx^7-\xx^6-2\xx^5-2\xx^2-\xx+1$ & $\sigma_4\sigma_5^2\sigma_4\deltfive^2$ \\
$s_6$& 1.88320 & $\xx^7-\xx^6-2\xx^4-2\xx^3-\xx+1$ & $\Delta_3\sigma_5^{-1}\sigma_4^{-1}$ \\
$s_7 \ *$& 2.17113 & $\xx^8-2\xx^7+\xx^6-4\xx^5+4\xx^4-4\xx^3+\xx^2-2\xx+1$ & $\Delta_3(\sigma_3\sigma_4\sigma_5)^{-2}$ \\
\hline
\end{tabular}
\medskip
\caption{Disc with $6$ punctures.}
\end{table}

\begin{table}[htbp]
\begin{tabular}{llll}
\hline
case & $\delta(s)$ & polynomial & braid \\
\hline
$s_1$& 1.55603 & $\xx^6-\xx^5-\xx^4+\xx^3-\xx^2-\xx+1$ & $\sigma_3\sigma_4\sigma_5\sigma_6\sigma_2\sigma_3\sigma_4\Delta_3\deltsix$ \\
$s_2$& 1.46557 & $\xx^7-2\xx^4-2\xx^3+1$ & $\sigma_4^{-2}\deltsix^2$ \\
$s_3$& 1.46557 & $\xx^7-2\xx^4-2\xx^3+1$ & $\sigma_6^2\deltsix^2$ \\
$s_4$& 1.55603 & $\xx^7-2\xx^5-2\xx^2+1$ & $\sigma_5^2\deltsix^3$ \\
$s_5 \ *$& 2.04249 & $\xx^7-2\xx^6-2\xx+1$ & $\sigma_4^{-2}\deltsix$ \\
$s_6$& 1.61094 & $\xx^8-\xx^7-2\xx^5+2\xx^4-2\xx^3-\xx+1$ & $\sigma_2^{-1}\sigma_3\sigma_4\sigma_5\deltsix^2$ \\
$s_7 \ *$& 2.47541 & $\xx^8-3\xx^7+2\xx^6-2\xx^5+2\xx^3-2\xx^2+3\xx-1$ & $\Delta_3\sigma_3(\sigma_3\sigma_4\sigma_5\sigma_6)^{-1}$ \\
$s_8 \ *$& 1.80979 & $\xx^8-\xx^7-2\xx^5-2\xx^3-\xx+1$ & $\Delta_4\sigma_6^{-1}\sigma_5^{-1}$ \\
$s_9$& 1.75488 & $\xx^8-\xx^7-4\xx^4-\xx+1$ & $\Delta_3\sigma_6^{-1}\sigma_5^{-1}\sigma_4^{-1}$ \\
$s_{10}$& 1.61094 & $\xx^9-\xx^7-2\xx^6-2\xx^3-\xx^2+1$ &
$\sigma_5^{-1}\sigma_4^{-1}\sigma_3\sigma_4\sigma_5\sigma_6\deltsix^3$ \\
$s_{11} \ *$& 2.04249 & $\xx^9-2\xx^8+\xx^7-2\xx^6-2\xx^3+\xx^2-2\xx+1$ & $\sigma_4\sigma_5\sigma_6\sigma_3\sigma_4\sigma_5\sigma_2^{-1}\sigma_1^{-1}\deltsix^{-1}$ \\
$s_{12} \ *$& 2.21497 & $\xx^{10}-2\xx^9-\xx^7-\xx^3-2\xx+1$ & $\sigma_2\sigma_1^2\sigma_2\deltsix^{-2}$ \\
\hline
\end{tabular}
\medskip
\caption{Disc with $7$ punctures.}
\end{table}


\providecommand{\bysame}{\leavevmode\hbox to3em{\hrulefill}\thinspace}
\providecommand{\MR}{\relax\ifhmode\unskip\space\fi MR }
\providecommand{\MRhref}[2]{%
  \href{http://www.ams.org/mathscinet-getitem?mr=#1}{#2}
}
\providecommand{\href}[2]{#2}

\end{document}